\documentclass[12pt]{article}
\usepackage{latexsym,amssymb,amsmath,amsfonts,amsthm,graphicx,url}
\newtheorem{theorem}{Theorem}

\theoremstyle{definition}

\def\beq{ \begin{equation} }
\def\eeq{ \end{equation} }
\def\mn{\medskip\noindent}

\def\square{\vcenter{\vbox{\hrule height .4pt
  \hbox{\vrule width .4pt height 5pt \kern 5pt
        \vrule width .4pt} \hrule height .4pt}}}

\def\ep{\epsilon}
\def\E{\mathbb{E}}
\def\RR{\mathbb{R}}

\def\ER{Erd\"os-R\'enyi }

\def\clearp{\clearpage}

\begin{document}

\title{The phase transition in Chung-Lu graphs}
\author{Aolan Ding and Rick Durrett
\thanks{This work was begin in a research independent study for AD in the Fall of 2022. RD was partially supported by NSF grant DMS 2011385
from the probability program.} \\
Dept.~of Math, Duke University 
}
\date{}

\maketitle	

\section{Abstract}

In 2002, Chung and Lu introduced a version of the Erdos-Renyi model which an edge between $i$ and $j$ is present with probability $p(i,j)$. They applied this model to compute the diameter of power-law random graphs, with yielded easier proofs than those for the configuration model \cite{NSW}. In 2007 their model was brought and integrated under the umbrella of Bolobas, Janson, and Riordan's inhomogeneous random graphs. However, the properties of the Chung-Lu model were never fully explored. In this paper, we fill the gap by giving a result for the cluster sizes in the subcritical regime and the fraction of vertices in the giant component in the supercritical phase.

\section{Introduction}

In the late 1950s, Erd\"os and R\'enyi introduced the random graph $G(n,p)$ in which there are $n$ vertices and, independently for each pair of vertices, there is an edge connecting them with probability $p$. This paper focuses on the sparse case $p=\lambda/n$ which has a phase transition as $\lambda$ varies. When $\lambda < 1$ (the subcritical case) the largest connected components is of order $\log n$, while if $\lambda>1$ (the supercritical case) there is a {\bf giant component} with $\sim \sigma(\lambda) >0$ vertices. These results are proved using the observation that when we randomly pick a vertex $x$ and define $Z_m$ as the number of vertices on the graph at distance $m$ from $x$, then when $m$ is small, $Z_m$ is a branching process in which each individual has a Poisson($\lambda$) number of children.

 At the end of the 20th century it was realized that in many real networks, such as the web pages in the world wide web, the computers that make up the internet, graphs of academic coauthors and networks of sexual contacts, vertices have {\bf power law degree distributions}: $p_k \sim C k^{-\gamma}$ where $2< \gamma < 3$. This observation motivated the construction of networks with a specified degree distribution, using for example the {\bf configuration model}. One begins with i.i.d. random variables $d_i$, $1\le i \le n$ with some distribution $p_k$, and conditioned to have $d_1+ \cdots + d_n$ even. One then attaches $d_i$ half-edges to $i$ and pairs the half-edges at random. If we pick a vertex $x$ at random and let $Z_m$ be the number of vertices on the graph at distance $m$ from $x$, then when $m$ is small, $Z_m$ is a {\bf two-phase branching process} in which the first generation has distribution $p$ and the second and subsequent generations have distributions $q$. 

The ancestor has $k$ children with probability $k$, but in all the future generations the number of children follows the {\bf size-biased distribution} $q_{j-1} = jp_j/\mu$ where $\mu = \sum_k k p_k$ and the term $j-1$ comes from the fact that one half-edge is used up making a connection from the previous generation. 
If we let the finite mean of $q$ be $\nu = \sum_j jq_j$ then the condition for a giant component in the configuration model is $\nu>1$. If $p_k \sim C k^{-\gamma}$ and $\gamma>3$ then the size biased distribution has finite mean. However if $2<\gamma<3$ then the size biased distribution has infinite mean and $Z_m$ grows doubly exponentially fast. 

In the configuration model, degrees are assigned initially. This makes the model difficult to study as when connections are made the degree distribution can deviate from the original. Chung and Lu took a different approach to avoid this issue. They assigned vertex $i$ a weight $w_i$ and independently introduced edges between vertex $i$ and vertex $j$ with probability
$$
p_{ij} = \frac{w_i w_j}{\sum_k w_k}.
$$
A condition is needed on the weights to guarantee that all the $p_{ij} < 1$ but this and all other technicalities will be ignored in the introduction. 
If we include self-loops from $i$ to $i$ then the mean degree of $i$
$$
Ed_i = \sum_j \frac{w_i w_j}{\sum_k w_k} = w_i.
$$
To have a power-law degree distribution $p_k \sim C k^{-\gamma}$ Chung and Lu assumed that
$$
w_i = \theta (i/n)^{-1/(\gamma-1)}.
$$
Using their model they were able to give a fairly simple proof that in the supercritical phase of these graphs with $\gamma>3$ and $\nu>1$ the distance between two random chosen vertices is $\sim \log n/\log\nu$ just as in the \ER case. In addition, they established an interesting result that when
$2< \gamma < 3$, the distance between two randomly chosen vertices on the giant component was asymptotically at most
$$
(2+o(1)) \log\log n/(-\log(\gamma-2)).
$$

Chung and Lu were much less successful with proving results in the subcritical case, see \cite{CLsub}. That situation improved dramatically when Bollobas, Janson, and Riordan \cite{BJR} defined a family of {\bf inhomogeneous random graphs} that contained the Chung-Lu model and a number of others as a special case. To define their models we will consider a special case sufficient for our needs, let $\kappa : (0,1) \times (0,1) \to \RR$ be continuous and symmetric $\kappa(x,y)= \kappa(y,x)$,  and define the probability of an edge between $i$ and $j$ to be
\beq
p(i,j) = \kappa(i/n, j/n) /n.
\eeq
In terms of the notation in \cite{BJR} we have taken $S=(0,1)$ and $\mu$ to be Lebesgue measure.

The Chung-Lu model is an example of the {\bf rank-1 case} in which 
\beq
\kappa(x,y) = \theta \psi(x)\psi(y)
\eeq
with $\psi(x) = x^{-1/(\gamma-1)}$. To  analyze the phase transition, let $T_k$ be the integral operator on $(S,\mu)$ with kernel $\kappa$, defined by
\beq
(T_k f)(x) = \int_{S} \kappa(x,y) f(y) d\mu(y),
\eeq
for any bounded measurable function $f$.
Defining the norm of this integral operator to be
$$
\lVert T_k \rVert = \sup\{ \lVert T_k f \rVert _2 : f \geq 0, \lVert f \rVert _2 \leq 1  \},
$$ 
we have the following theorem from Bollobas, Janson, and Riordan \cite{BJR}

\begin{theorem}
 If $\lVert T_k \rVert < 1$, then the largest component $C_1$ has size $o(n)$. If $\lVert T_k \rVert > 1$, then the largest component has size $\Theta(n)$.
\end{theorem}

\noindent
Here $a_n=o(n)$ if $a_n/n \to 0$ while $b_n=\Theta(n)$ if there are constants 
$0 < B_1 < B_2 < \infty$ so that $B_1 < b_n/n < B_2$ for large $n$.

In the rank-1 case, the operator $T_\kappa$ has 
$$
T_k f = \theta \biggl (\int f \psi d\mu \biggl) \psi.
$$
Since the operator $T_\kappa$ maps all functions $f$ to multiples of $\psi$, $\psi$ is the only  eigenvector. It has eigenvalue $\theta \int \psi(x)^2 dx$ so the critical critical value
$$
\theta_c = \frac{1}{\int \psi(x)^2 dx}.
$$ 
Therefore, in the Chung-Lu model, 
$$
\theta_c = \begin{cases} 0 &\hbox{ if $\gamma \leq 3$} \\
>0 & \hbox{ if $\gamma >3$} \end{cases}
$$ 

To calculate the size of the giant component, we define a nonlinear operator by
$$
\Phi_\kappa f = 1 - \exp(-T_\kappa f).
$$
The survival probability $\rho_\kappa(x)$ is the maximum fixed point of $\Phi_\kappa$ so we have
\beq
\rho_{\kappa}(x) = 1 - \exp\left( - \int_0^1 \kappa(x,y) \rho_\kappa(x) \, dy \right).
\label{rhoeq}
\eeq
In the \ER case $\kappa(x,y) \equiv \lambda$ so $\rho_\kappa(x) \equiv \rho$ and
the equation becomes
$$
1 - \rho = \exp(-\lambda \rho )
$$ 
If we let $\zeta = 1- \rho$ be the extinction probability then this becomes
$$
\zeta = \exp( - \lambda (1-\zeta))
$$
a well-known result for the \ER case: when $\lambda>1$ the extinction probability is the fixed point in $(0,1)$ of the degree distribution, which is Poisson($\lambda$).

To state the next result which gives the size of the largest component $|{\cal C}_1|$, we need a definition. Define $\kappa$ is {\bf reducible} if there is an $A$ with $0 < \mu(A) \le \mu({\cal S})$ so that $\kappa(x,y)=0$ on $A \times ({\cal S}-A)$. Otherwise $\kappa$ is {\bf irreducible}.

\begin{theorem} \label{BJR3.1}
 Suppose $\|T_\kappa\| > 1$ and $\kappa$ is irreducible then
$$
\frac{|{\cal C}_1|}{n} \to \bar\rho = \int_0^1 \rho_\kappa(x) \, dx.
$$
\end{theorem}

\mn
To explain this, the probability that $x$ is in the giant component is the same as $\rho_\kappa(x)$ the probability that starting from $x$ does not die out, so the integral gives the fraction of vertices in the giant component. 

For the Chung-Lu model we can calculate $\bar\rho$ using \eqref{rhoeq}. This result, proved in Section 2, is

\begin{theorem} \label{CLgc}
As $\theta \downarrow \theta_c$ (which is 0 if $\gamma \le 3$)
$$
\bar\rho(\theta) \sim \begin{cases}
C_\gamma \theta^{1/(\gamma-3)} & 2 < \gamma < 3 \\
C_3\exp( - 1/2\theta) & \gamma = 3 \\
C_\gamma  (\theta-\theta_c) & \gamma > 3
\end{cases}
$$
\end{theorem}

\mn
The result for $\gamma=3$ is due to Riordan \cite{riordan2005}.

Our final result concerns the size of clusters in Chung-Lu model the subcritical regime. 
$\theta_c=0$ when $\gamma \le 3$ so we only have to consider $\gamma>3$.

\begin{theorem} \label{CLsubcl}
In the Chung-Lu model with $\gamma>3$, for any $\theta<\theta_c$, there exists a constant $C$ which depends on $\gamma$ and $\theta$ so that 
$$
P \left( \max_{v} |{\cal C}_v| \geq C n^{-1/(\gamma-1)} \right) \rightarrow 0
$$ 
\end{theorem}

\noindent
The lower bound on the size of the largest cluster can be easily calculated: since $P( d_i > k) \sim C k^{-(\gamma-1)}$ we have
$$
\max_{1\le i \le n} d_i = O( n^{-1/(\gamma-1)} )
$$
Since the cluster containing $i$ is larger than $d_i$, the lower bound follows.
To prove the upper bound, we have to estimate the probability that a vertex has a cluster that is much larger than its degree. The proof parallels Janson's proof \cite{jansonsub} for the configuration model but is easier since the presence of edges in the inhomogeneous random graph are independent events, which is not true in the configuration model. The remaining sections of the paper is devoted to proofs, which Theorem 3 is proved in Section 2, and Theorem 4 in Section 3.

\section{Survival probabilities}

In this section we will prove Theorem \ref{CLgc} when $2<\gamma < 3$ and $\gamma > 3$.
The first step is to observe that if $a$ is fixed, then as $n\to\infty$ the vertices $b$ that are connected to $a$ are a {\bf Poisson process with mean measure} $\kappa(a,b) \, db$ on $(0,1]$. That is, if we let $N(A)$ be the number of points in $A$ then

\mn
(i) If $A_1, \ldots A_k$ are disjoint, then $N(A_1), \ldots N(A_k)$ are independent

\mn
(ii) $N(A)\hbox{ has Poisson}\left(\int_A \kappa(y,z) \, dz \right)$ distribution.

\mn
To make it easier to compare with Riordan's work \cite{riordan2005} we will change variables $\rho_\kappa(a) = S_\infty(a)$.

If the initial type is $a$, then the expected number of first generation children that survive is
$$
\mu = \int_0^1 \kappa(a,b) S_\infty(b) \, db
$$
The set of offspring is a Poisson process on $(0,1]$. If we thin the Poisson process by keeping only the points that have giant components, the result is a Poisson process with mean measure $\kappa(a,b) S_\infty(b) \, db$.

If the number of neighbors of $a$ that have a giant component has mean $\mu$ then the probability probability of at least one has a giant component is $1 -e^{-\mu}$ so
\beq
S_\infty(a) = 1 - \exp\left( - \int_0^1 \kappa(a,b)  S_\infty(b) \, db \right)
\label{Saeq}
\eeq
Specializing to the Chung-Lu model $\kappa(a,b) = \theta \psi(a)\psi(b)$,
with  $\psi (x) = x^{-1/(\gamma -1)}$.
If we let $A_\theta =  \int_0^1 \theta \psi(b) S_\infty(b) \, db$ then
\beq
S_\infty(a) = 1 - \exp(-A_\theta \psi(a))
\label{Saeq2}
\eeq
so using \eqref{Saeq} and \eqref{Saeq2}, equating the exponentials, dividing by $\psi(a)$, and multiplying by $\sqrt{a}$ we have
\beq
A_\theta = \theta \int_0^1 \psi(b) [1 - \exp(-A_c \psi(b) ] \, db
\label{Aceq}
\eeq
If we let
\beq 
g(x) = \int_{0}^{1} \psi(b) [1-\exp(-x \psi(b))] \, db \label{gx}
\eeq
Then, we have
\beq
A_\theta = \theta g(A_\theta)
\eeq

\mn
{\bf Case 1.} $\gamma>3$
\beq
\int_{0}^{1} \psi(x) ^2 dx < \infty
\eeq
and the critical value $\theta_c = 1/ \int_{0}^{1} \psi(x) ^2 dx$. Using \eqref{gx} we have 
\beq
g(x) \sim x \int_{0}^{1} \psi(y) ^2 dy
\eeq
as $x\to 0$. Differentiating equation (\ref{gx})
\beq
g'(x) \sim \int_{0}^{1} \psi(b) ^2 \exp(-x \psi(b)) db
\eeq
where $g$ has a finite slope at 0. In this case, since $x$ is concave, as $x$ increases, $g'(x)$ decreases. 

Since $g(0) = 0$, $g'(0) = \int_{0}^{1} \psi(x) ^2 dx $, and $g(x) \rightarrow \int_{0}^{1} \psi(b) db$ as $x \rightarrow \infty$, we have a solution of $x/\theta =g(x)$ only when $1/\theta < \int_{0}^{1} \psi(x)^2 dx$, that is, when $\theta>\theta_c$. Expanding $g$ in power series about 0 we have
\beq
g(x) = x g'(0) + \frac{x^2}{2} g''(0) + ...
\eeq
and thus
\beq
A_\theta = \theta g(A_\theta) \approx \theta g'(0) A_\theta + g''(0) \frac{A_{\theta}^2}{2}
\eeq
substituting $\theta_c = 1/g'(0)$ into the above equation, we can calculate $A_\theta$ as
\beq
A_\theta = \frac{2(\theta g'(0) -1)}{-g''(0)} = \frac{2g'(0) }{-g''(0)}(\theta-\theta_c)
\eeq
Since $S_{\infty} (a) = 1- \exp (- A_\theta \psi(a)) \sim \psi(a)A_\theta$
as $A_\theta$ approaches 0. Thus as $\theta$ approaches $\theta_c$, 
\beq
S_{\infty} (a) \sim C\psi(a) (c-c_c)
\eeq
proving the conclusion for $\gamma>3$.

\mn
{\bf Case 2.} When $2< \gamma <3$, we have
\beq
g(x) = \int_{0}^{1} \frac{1}{b^{1/(\gamma -1)}} (1-\exp(-x/b^{1/(\gamma-1)})) db
\eeq
We set $z = b^{1/(\gamma -1)}$ and $dz = 1/(\gamma -1) b^{1/(\gamma-1) -1}$, which subsequently gives
\beq
db = (\gamma - 1) b^{-1/(\gamma -1) + 1} = (\gamma - 1) z^{-1+(\gamma - 1)} dz = (\gamma -1) \frac{dz}{z^{2-\gamma}}
\eeq
Thus, we have
\beq 
g(x) = (\gamma -1) \int_{0}^{1} \frac{1}{z^{3-\gamma}}(1-\exp(-x/z)) dz
\eeq
Let $z = x/y$, which means $dz = -(x/y^2)\, dy$ to get 
\beq
g(x) = x^{\gamma - 2} \int_{x}^{\infty} y^{1-\gamma} (1-e^{-y})dy
\eeq

Since we have $\gamma >2$, we have $1-\gamma < -1$. Substitute into the above equation, we can see that the integral is convergent at $\infty$.  Since we have $\gamma <3$, we have $2-\gamma < -1$ , which the integral is convergent at 0 and 
\beq
g(x) \sim C_{\gamma} x^{\gamma -2}
\eeq
where
\beq 
C_{\gamma} = \int_{0}^{\infty} y^{1-\gamma} (1-e^{-y})dy
\eeq
As $x$ approaches $0$, $x/g(x) \sim x^{3-\gamma}/C_{\gamma}$. Since $\theta = A_\theta / g (A_\theta)$, we have $\theta \sim A_\theta^{3-\gamma} / C_{\gamma}$ which implies
\beq
A_c \sim (C_\gamma x)^{1/(3-\gamma)}
\eeq
and completes the proof for $2<\gamma<3$.

\section{Subcritical cluster sizes}

In this section we continue to  investigate the Chung-Lu model with $\kappa (x,y) = \theta \psi(x) \psi (y)$ where $\psi(x) = x^{-1/(\gamma-1)}$. Since $\theta_c=0$ for $\gamma > 3$, we now restrict our attention to $\gamma>3$. In this case $\int \psi(y)^2 \, dy < \infty$ so 
$$
\theta_c = 1/ \int_0^1  \psi(y)^2 dy > 0.
$$  
The Cauchy-Schwarz inequality implies that
$$
B =\int \psi(u) \, du < \left( \int  \psi(y)^2 dy \right)^{1/2} < \infty
$$ 
When $B<\infty$ the set of neighbors of the vertex $y$ can be constructed in two steps.

\mn
(i) The total number of neighbors is $N(y)=$ Poisson with mean 
$$
\int \kappa(y,z) \, dz = \theta\psi(y) B. 
$$
(ii) Conditional on $N(y)=m$, the labels of the $m$ neighbors are distributed according to  
\beq
\bar\psi(z) = \psi(z)/B.
\eeq

\medskip
Given these observations, we can construct the {\bf exploration process} in which one reveals the neighbors of vertices one at a time to get a random walk as follows: if we start from $X_0=x_0$, then the number of neighbors $N(x_0) =$ Poisson with mean $\theta\psi(x_0)B$. Since $A_0=1$ we have
$$
A_1 = N(x_0)
$$
For all $t\ge 1$ the vertex $X_t$ whose neighbors are added on the $t$-th step is distributed according to $\bar\psi(z)$, and
$$
A_{t+1} = A_t  -1 + N(X_t)
$$
The distribution of $N(X_t)$ is a mixture
$$
\int_0^1 \bar \psi(y) \, \hbox{Poisson}(c\psi(y)B) \, dy
$$

To compute the distribution of $N(X_t)$ we use the methods of Section 13 in \cite{BJR}. To see the answer easily, we think of $\text{Poisson}(\theta\psi(y)B)$ as being a point mass at the mean. This can be justified by noting that the standard deviation of Poisson($\lambda$) is $\sqrt{\lambda} \ll \lambda$. Observe that $c\psi(y) B \geq x$ when 
$$
y^{-1/{\gamma-1}}\geq x/cB \hspace{1cm}\text{or} 
\hspace{1cm} y \leq (x/cB)^{-(\gamma-1)}
$$ 
and letting $z = (x/cB)^{-(\gamma-1)}$, we have 
\beq
P(N(X_t) \geq x) = \int_{0}^{z} \frac{\psi(y)}{B} dy = \frac{C}{B} x^{1-1/(\gamma-1)} = C' x^{-(\gamma-1)+1} \label{3.6.13}
\eeq
where in this section $C$ is a constant that changes from line to line. Here, we have indicated that the constant changed was made by writing $C'$.

\begin{proof} [Proof of Theorem \ref{CLsubcl}.]
At time $\tau = \inf\{ t : A_t=0 \}$ we have found the $\tau$ members of the cluster.
The exploration process is easier to analyze if it is homogeneous in time, so we will
suppose that $X_0$ is distributed according to $\bar\psi(z)$.  To facilitate the comparison with Janson \cite{jansonsub}, we rewrite the recursion $A_{t+1} = A_t -1 + N(X_t)$ as 
$$
S_{t+1} = S_t - 1 + \xi_t
$$
 If $|{\cal C}(x_0)| \geq M$, then $\tau \geq M$ and 
    \beq
    S_M = d(x_0) + \sum_{i=1}^{M} (\xi_i -1) 
    \label{S_M}
    \eeq
Since $\theta < \theta_c$, we have $\E[\xi_i] = v <1$. If $A$, which will be chosen later, is large enough and $\delta < 1/(\gamma-1)$, we let 
\beq
M = An^{1/(\gamma -1)} \qquad\hbox{and}\qquad
 M_1 = n^{1/(\gamma-1) - \delta}
\label{MM1}
\eeq
where $M$ will be the upper bound for the largest cluster size $|{\cal C}^1|$. Suppose we define the truncated random variables $Y_i = \xi_i 1_{(\xi_i \leq M_1)}$, then we have 
$$
\E[Y_i] \leq \E[\xi_i] = v < 1.
$$ 
If $|{\cal C}(x_0)| > M$, then (\ref{S_M}) holds. Since $S_M \ge 0$, \eqref{S_M} implies
$$
M \le  d(x_0) + \sum_{i=1}^{M} \xi_i
$$
Using  $\xi_i \leq \Delta = \max_i d_i$, the above implies
\begin{align} 
M & \leq \Delta + \sum_{i=1}^M \xi_i 1_{(\xi_i \leq M_1)} 
+ \sum_{i=1}^M \xi_i 1_{(\xi_i > M_i)}
\nonumber\\
    & \leq \sum_{i=1}^{M} Y_i + \Delta \biggl( 1+ \sum_{i=1}^{M} 1_{(\xi_i > M_i)}\biggl)\\
 & \leq vM +\sum_{i=1}^M (Y_i - \E Y_i) + \Delta \biggl( 1+\sum_{i=1}^M 1_{(\xi_i > M_i)}\biggl)
\nonumber
\end{align}
If $v< 1-2\epsilon$ then on $|{\cal C}(x_0)| > M$ we have
$$
2\ep M \leq \sum_{i=1}^M (Y_i - \E Y_i) 
+ \Delta \biggl( 1+\sum_{i=1}^M 1_{(\xi_i > M_i)}\biggl)
$$
From this we conclude that
\beq
P(|{\cal C}(x_0)| \geq M) \leq P \biggl(\sum_{i=1}^M |Y_i - \E Y_i| > \epsilon M\biggl) + P\biggl(\sum_{i=1}^{M}1_{(X_i > M_i)} > \epsilon \frac{M}{\Delta} -1\biggl) \label{3.6.17}
\eeq
To bound the first sum in equation (\ref{3.6.17}), we use Rosenthal's inequality \cite{rosenthal} for non-negative random variables: 
\beq
\E\left(\sum_{k=1}^n Z_i\right)^ r 
\leq C_r \max \Biggl(\biggl( \sum_{k=1}^n \E Z_{i}^2\biggl)^{r/2}, \sum_{k=1}^n \E Z_{i}^r\Biggl)
\label{Rosen}
\eeq
for $r \ge \gamma$ so that $r\delta > 2$. Letting $Z_i = |Y_i -\E Y_i|$, and $n=M$ we have 
$$
 \E \biggl | \sum_{i=1}^M |Y_i - \E Y_i|\biggl|^r 
 \leq C_r M^{r/2} (\E |Y_1 - \E Y_1|^2)^{r/2} + C_r M \E |Y_1 - \E Y_1|^r 
$$
In the first term we used the fact that $\E (X-c)^2$ is minimized when $c = \E X$. For the second term, suppose we write $\E (Z;A)$ for the integral of $Z$ over $A$
\begin{align*}
    \E|Y_1-\E Y_1|^r & = \E(|Y_1 - \E Y_1|^r ; Y_1 \geq \E Y_1) 
+ \E (|Y_1-\E Y_1|^r; Y_1< \E Y_1)\\
    & \leq \E(|Y_1|^r; Y_1 \geq \E Y_1) + |\E Y_1|^r P(Y_1<\E Y_1) \leq 2 \E|Y+1|^r
\end{align*}
because $Y_1 \geq 0$, and Jensen's inequality can be used to put the $r$-th power inside the expected value
\beq
 \E \biggl | \sum_{i=1}^M |Y_i - \E Y_i|\biggl|^r 
\leq C_r M^{r/2} (\E |Y_1|^2)^{r/2} + C_r M \E |Y_1 |^r.
\label{Rosen2}
\eeq

The second moment of $Y_1$ can be estimated using \eqref{3.6.13} when $\gamma>3$
\begin{align}
    \E Y_1^2 & = \int_{0}^{\infty} 2x P(Y_1 > x) \, dx 
  = \int_{0}^{M_1} 2x P(\xi_1 > x) dx
\nonumber \\
& \leq 1+ C \int_{1}^{M_1} 2x^{3-\gamma} dx
 \leq C M_1 \label{3.6.19}
\end{align}
Trivially, $\E Y_1^r \leq M_1^r$. Using Markov's inequality, the second moment estimate above, $Y_1 \leq M_1$, and $M_1 / M = A^{-1} n^{-\delta}$, we have 
\begin{align}
    P\biggl( \sum_{i=1}^{M}|Y_i - \E Y_i| > \epsilon M \biggl) 
& \leq (\epsilon M)^{-r} \E \biggl | \sum_{i=1}^M |Y_i - \E Y_i|\biggl|^r
\nonumber\\
    & \leq C (\E Y_1^2 / M)^{r/2} + C  M^{1-r} \E Y_1^r
 \label{3.6.20}\\
&   \leq C' (M_1 / M) ^{r/2} + C' M (M_1 / M)^r \\
    & \leq C'' n^{-r \delta / 2} + C'' n^{1-r\delta}= o(n^{-1})
\nonumber
\end{align}
since $r\delta >2$

For the second sum in (\ref{3.6.17}), we write $I_i = 1_{\xi_i > M}$ and we have 
\begin{align*}
    P\biggl( \sum_{i=1}^M I_i \geq L \biggl) 
& \leq {M \choose L} P (I_k = 1, 1\leq k \leq L)
     = {M \choose L} P(I_1 = 1)^L\\
    & = {M \choose L} P(X > \xi_i)^L \leq (MP(\xi_i > M_1))^ L
\end{align*}
By the tail bound on $\xi_i$ in (\ref{3.6.13}) and the choices of $M, M_1$ in  \eqref{MM1}, we have
\begin{align*}
    MP(\xi_1 > M_1) & \leq M \cdot  C M_1^{-(\gamma - 1)+1}\\
    & \leq C An^{1/(\gamma - 1) + (2-\gamma)[1/(\gamma-1) -\delta]}\\
    & = C A_n^{(\gamma-2)\delta - (\gamma-3)/(\gamma-1)}
\end{align*}
If we choose $\delta>0$ so that 
$\delta_1 = (\gamma -3)/(\gamma -1) - (\gamma -2)\delta > 0$ then 
\beq
MP (\xi_1 > M_1) = O(n^{-\delta_1})\quad\hbox{ and} 
\quad
P\biggl(\sum_{i=1}^{M} I_i \geq L \biggl) = O(n^{-L \delta_1}) \label{3.6.22}
\eeq
If we choose $L> 1/\delta_1$, then the last error term is $o(n^{-1})$.  If $A$ is chosen large enough, 
\beq
\epsilon \frac{M}{\Delta} -1 > L \text{ for large $n$} \label{3.6.23}
\eeq
with (\ref{3.6.20}) and (\ref{3.6.22}) we have shown that the right-hand side of (\ref{3.6.17}) is $o(n^{-1})$), and the proof of Theorem 4 is complete. 

\end{proof}

\clearp

\end{document}